\documentclass[a4paper]{amsart}
\usepackage{latexsym,bm,stmaryrd}
\usepackage[a4paper,hmargin=26mm,vmargin=28mm]{geometry}
\usepackage{latexsym,bm}
\usepackage[svgnames]{xcolor}
\usepackage{xparse}
\usepackage{listings}
\usepackage{tikz}
\usetikzlibrary{matrix}
\usepackage[utf8]{inputenc}
\usepackage[english]{babel}
\usepackage[T1]{fontenc}
\usepackage{amsmath, amsfonts, amssymb, amsthm, mathrsfs, pb-diagram}
\usepackage{lmodern}
\usepackage{fullpage}
\usepackage{microtype}
\usepackage{enumerate}
\usepackage{mathtools}
\usepackage{enumitem}
\setlist[enumerate]{label=\alph*\upshape), nolistsep}

\usepackage{etoolbox}
\usepackage{aliascnt}

\usepackage[pagebackref=false]{hyperref}

\newcommand\enumref[2]{\hyperref[#2]{\autoref*{#1}(\autoref*{#2})}}

\usepackage[misc]{ifsym}

\usepackage{young}
\usepackage{youngtab}
\usepackage{ytableau}

\def\NewTheorem#1{%
  \newaliascnt{#1}{equation}%
  \newtheorem{#1}[#1]{#1}%
  \aliascntresetthe{#1}%
  \expandafter\def\csname #1autorefname\endcsname{#1}%
}

\let\gdom=\vartriangleright

\DeclareMathOperator\Shape{Shape}

\let\<=\langle
\let\>=\rangle
\def\({\big(}
\def\){\big)}
\def\A{\mathbb{A}}
\def\Z{\mathbb{Z}}
\def\Q{\mathbb{Q}}

\def\N{\mathbb{N}}

\def\t{\mathfrak{t}}
\def\s{\mathfrak{s}}
\def\u{\mathfrak{u}}
\def\v{\mathfrak{v}}
\def\A{\mathcal{A}}

\def\O{\mathcal{O}}

\def\lam{\lambda}

\def\Sym{\mathfrak{S}}
\def\eps{\varepsilon}
\def\fm{\mathfrak{m}}
\def\fn{\mathfrak{n}}
\newcommand\HH{\mathscr{H}}

\def\P{\mathscr{P}}

\def\mff{\mathfrak{f}}

\def\mfg{\mathfrak{g}}

\def\dmn{\mathcal{D}_m(n)}

\let\Gdom=\blacktriangleright
\DeclareMathOperator\Gedom{\,{\underline{\kern-.1ex{\blacktriangleright}\kern-0.1ex}}\,}

\DeclareMathOperator\RS{RSK}

\DeclareMathOperator\Std{Std}

\numberwithin{equation}{section}
\newtheorem{prop}[equation]{Proposition}
\newtheorem{thm}[equation]{Theorem}
\newtheorem{cor}[equation]{Corollary}

\newtheorem{lem}[equation]{Lemma}

\theoremstyle{definition}
\newtheorem{dfn}[equation]{Definition}
\theoremstyle{remark}
\newtheorem{rem}[equation]{Remark}
\newtheorem{example}[equation]{Example}

\PassOptionsToPackage{unicode}{hyperref}
\PassOptionsToPackage{naturalnames}{hyperref}

\begin{document}

\title[Kazhdan-Lusztig left cell preorder and dominance order]
{Kazhdan-Lusztig left cell preorder and dominance order}

\author{Zhekun He}\address{School of Mathematical and Statistics\\
	Beijing Institute of Technology\\
	Beijing, 100081, P.R. China}
\email{hzk3002@163.com}

\author{Jun Hu}\address{School of Mathematical and Statistics\\
	Beijing Institute of Technology\\
	Beijing, 100081, P.R. China}
\email{junhu404@bit.edu.cn}

\author[Corresponding author]{Yujiao Sun\textsuperscript{\Letter}}\thanks{\Letter Yujiao Sun \qquad Email: yujiao.sun@bit.edu.cn}
\address{School of Mathematical and Statistics\\
	Beijing Institute of Technology\\
	Beijing, 100081, P.R. China}
\email{yujiao.sun@bit.edu.cn}

\begin{abstract} Let ``$\leq_L$'' be the Kazhdan-Lusztig left cell preorder on the symmetric group $\Sym_n$. Let $w\mapsto (P(w),Q(w))$ be the Robinson-Schensted-Knuth correspondence between $\Sym_n$ and the set of standard tableaux with the same shapes. We prove that for any $x,y\in\Sym_n$, $x\leq_L y$ only if $Q(y)\unrhd Q(x)$, where ``$\unrhd$'' is the dominance (partial) order between standard tableaux. As a byproduct, we generalize an earlier result of Geck by showing that each Kazhdan-Lusztig basis element $C'_w$ can be expressed as a linear combination of some $\fm_{\u\v}$ which satisfies that $\u\unrhd P(w)^*$, $\v\unrhd Q(w)^*$, where $\t^*$ denotes the conjugate of $\t$ for each standard tableau $\t$, $\{\fm_{\s\t} \mid \s,\t\in\Std(\lam),\lam\vdash n\}$ is the Murphy basis of the Iwahori-Hecke algebra $\HH_{v}(\Sym_n)$ associated to $\Sym_n$.
\end{abstract}

\maketitle
\setcounter{tocdepth}{1}

\section{Introduction}

Let $(W,S)$ be a Coxeter system (\cite[Chapter 3]{B}). The Iwahori-Hecke algebra $\HH_v(W)$ associated to $(W,S)$ is a free $\Z[v,v^{-1}]$-module with basis $\{T_w\mid w\in W\}$ and multiplication determined by
$$
\begin{aligned}
T_s^2 &=(v-v^{-1})T_s+1, \forall\,s\in S, \\
T_wT_u &=T_{wu}, \quad\text{if $\ell(wu)=\ell(w)+\ell(u)$}.
\end{aligned}
$$

In their influential paper \cite{KL}, Kazhdan and Lusztig introduced the famous KL bases $\{C'_w \mid w\in W\}$ and KL polynomials $\{P_{y,w}\mid y,w\in W\}$ for the Iwahori-Hecke algebra $\HH_v(W)$. Though the definitions of these basis elements $C'_w$ and polynomials $P_{y,w}$ are elementary, deep connections with the geometry (such as intersection cohomology and perverse sheaves) were found. In particular, many properties of these bases and polynomials can only be proved via the categorification of Hecke algebra using perverse sheaves. Ever since, KL bases and KL polynomials play an important role in the representation theory of Lie algebras, algebraic groups, finite groups of Lie type and quantum groups.

Kazhdan and Lusztig used these bases $\{C'_w\mid w\in W\}$ to define three preorders $\leq_L, \leq_R$ and $\leq_{LR}$ on the Coxeter group $W$, which are called left cell preorder, right cell preorder and two-sided preorder. The associated equivalence classes are called the left cells, right cells and two-sided cells of $W$ respectively. We use the notation $x\sim_L y$ to mean that $x$, $y$ are in the same KL left cell. Similarly, we have the notations $x\sim_R y$ and $x\sim_{LR}y$. In the case $W$ is a Weyl group, these notions play an significant role in the study of representations of finite groups of Lie type and BGG category $\O$ of complex semisimple Lie algebras. However, as with the KL polynomials, there are no elementary combinatorial characterizations (without using KL bases) for these preorders and cells except in some special cases.

In this paper we focus on the symmetric group $\Sym_n$ on $\{1,2,\cdots,n\}$. Let $\RS\colon \Sym_n\rightarrow\sqcup_{\lam\vdash n}(\Std(\lam)\times\Std(\lam)), w\mapsto (P(w),Q(w))$ be the Robinson-Schensted-Knuth (bijective) correspondence (\cite[22.1.B]{B}, \cite[\S4.1]{F}), where $\lam\vdash n$ means $\lam$ is a partition of $n$, and $\Std(\lam)$ denotes the set of standard $\lam$-tableaux. With this combinatorial models, we have the following beautiful characterization of KL cells.

\begin{thm}[{\cite{KL}, \cite{V}, \cite{Du}, \cite{A}}] Let $x,y\in\Sym_n$. Suppose that $\RS(x)\in\Std(\lam)\times\Std(\lam)$, $\RS(y)\in\Std(\mu)\times\Std(\mu)$, where $\lam\vdash n, \mu\vdash n$. Then \begin{enumerate}
\item[1)] $x\sim_L y$ if and only if $Q(x)=Q(y)$;
\item[2)] $x\sim_R y$ if and only if $P(x)=P(y)$;
\item[3)] $x\sim_{LR} y$ if and only if $\lam=\mu$.
\end{enumerate}
\end{thm}

It is natural to ask if one can give some characterization of the Kazhdan-Lusztig preorders $\leq_L, \leq_R$ and $\leq_{LR}$ on the symmetric group $\Sym_n$ in terms of the combinatorics of standard tableaux via the Robinson-Schensted-Knuth correspondence $\RS$. Unfortunately, to the best of our knowledge, little is known in this direction except the following well-known result.

\begin{prop}[{\cite{Shi}, \cite{LX}, \cite[(2.13.1)]{DPS}, \cite[Theorem 5.1]{Gec}}] Let $x,y\in\Sym_n$. Suppose that $\RS(x)\in\Std^2(\lam)$, $\RS(y)\in\Std^2(\mu)$. Then $x\leq_{LR}y$ if and only if $\mu\unrhd\lam$.
\end{prop}

The following theorem is {\bf the first main result} of this paper, which can be viewed as a first step in the combinatorial characterization of the left cell and right cell preorders
$\leq_L, \leq_R$.

\begin{thm}\label{mainthm1} Let $x,y\in\Sym_n$. If $x\leq_{L}y$, then $Q(y)\unrhd Q(x)$. Equivalently, if $x\leq_{R}y$, then $P(y)\unrhd P(x)$.
\end{thm}

\begin{rem}
In general, the converse of Theorem \ref{mainthm1} does not hold.
For example, in $\mathfrak{S}_3$, we have that
$$Q(s_2)=
\ytableausetup{textmode}
\begin{ytableau}
 1 & 2  \\
 3 \\
\end{ytableau}
\rhd
\ytableausetup{textmode}
\begin{ytableau}
 1 & 3  \\
 2 \\
\end{ytableau}
= Q(s_1),$$
but $s_1\nleq_{L} s_2$. Similarly, in $\mathfrak{S}_4$, we have that
$$Q(s_2)=
\ytableausetup{textmode}
\begin{ytableau}
 1 & 2 & 4 \\
 3 \\
\end{ytableau}
\rhd
\ytableausetup{textmode}
\begin{ytableau}
 1 & 3  \\
 2 & 4  \\
\end{ytableau}
= Q(s_1s_3),$$
but $s_1s_3\nleq_{L} s_2$. Moreover, even we combine the dominance condition with the inverse inclusion condition of their right descent sets, it is still not sufficient to imply the left cell preorder $\leq_L$. In fact, in $\mathfrak{S}_5$, we have that
$$Q(s_2s_3s_4s_3s_2)=
\ytableausetup{textmode}
\begin{ytableau}
 1 & 2 & 4 \\
 3 \\
 5 \\
\end{ytableau}
\rhd
\ytableausetup{textmode}
\begin{ytableau}
 1 & 4  \\
 2 & 5  \\
 3 \\
\end{ytableau}
= Q(s_4s_1s_2s_1)$$
and
$$
\mathcal{R}(s_2s_3s_4s_3s_2)
=\{s_2,s_4\}
\subseteq
\{s_1, s_2,s_4\}
=\mathcal{R}(s_4s_1s_2s_1),
$$
but $s_4s_1s_2s_1 \nleq_{L} s_2s_3s_4s_3s_2$ because one can calculate that $D'_{s_4s_1s_2s_1}C'_{s_2s_3s_4s_3s_2 }=0$, where $\mathcal{R}(w):=\{s_k\mid 1\leq k<n, ws_k<w\}$ is the right descent set of $w$, and $D'_w$ is the dual Kazhdan-Lusztig basis element as defined in \cite[(5.1.7)]{Lu0}.
\end{rem}

Our approach to the proof of the above theorem is to investigate some subtle relationship between the Kazhdan-Lusztig basis $\{C'_w\mid w\in\Sym_n\}$ and the seminormal basis $\{\mff_{\s\t}\mid \s,\t\in\Std(\lam),\lam\vdash n\}$ of the semisimple Hecke algebra $\HH_{\Q(v)}(\Sym_n):=\Q(v)\otimes_{\Z[v,v^{-1}]}\HH_v(\Sym_n)$. The seminormal basis $\{\mff_{\s\t}\mid \s,\t\in\Std(\lam),\lam\vdash n\}$ is a nice basis of $\HH_{\Q(v)}(\Sym_n)$ which can be used to construct matrix units of the Hecke algebra and only exists inside the semisimple Hecke algebra $\HH_{\Q(v)}(\Sym_n)$. As a byproduct, we generalize an earlier result of Geck \cite{Gec} on the relationship between the Kazhdan-Lusztig basis $\{C'_w\mid w\in\Sym_n\}$ and the Murphy basis $\{\fm_{\s\t}\mid \s,\t\in\Std(\lam),\lam\vdash n\}$ of $\HH_v(\Sym_n)$. To state our second main result, we first recall the following two different partial orders on $\Std^2(n):=\sqcup_{\lam\vdash n}\Std(\lam)\times\Std(\lam)$. Let ``$\rhd$'' be the dominance (partial) order defined on the set of partitions of $n$ and also on the set of standard tableaux (see Section 2 for detailed definitions).

\begin{dfn} Let $\lam,\mu\vdash n$, $(\s,\t)\in\Std(\lam)\times\Std(\lam)$ and $(\u,\v)\in\Std(\mu)\times\Std(\mu)$.  We define $$
\text{$(\s,\t)\unrhd (\u,\v)$ if either $\lam\rhd\mu$ or $\lam=\mu$, $\s\unrhd\u$ and $\t\unrhd\v$}.
$$
If $(\s,\t)\unrhd (\u,\v)$ and $(\s,\t)\neq (\u,\v)$, the we write $(\s,\t)\rhd(\u,\v)$. We also define $$
\text{$(\s,\t)\Gedom (\u,\v)$ if $\s\unrhd\u$ and $\t\unrhd\v$}.
$$
If $(\s,\t)\Gedom (\u,\v)$ and $(\s,\t)\neq (\u,\v)$, the we write $(\s,\t)\Gdom(\u,\v)$.
\end{dfn}

In particular, if $(\s,\t)\Gedom (\u,\v)$ then $\Shape(s)=\Shape(\t)\unrhd\Shape(\u)=\Shape(\v)$. It follows that the partial order ``$\Gedom$'' is stronger than the partial order ``$\unrhd$'' on $\Std(\lam)\times\Std(\lam)$. In order to simplify the notations, we henceforth make the following definition on notations.

\begin{dfn} Let $w\in\Sym_n$. We define $$
\u_w:=P(w)^*,\quad \v_w:=Q(w)^* ,
$$
where $\t^*$ denotes the conjugate of $\t$ for each standard tableau $\t$.
\end{dfn}

In addition to the Kazhdan-Lusztig basis $\{C'_w\mid w\in\Sym_n\}$, the Hecke algebra $\HH_v(\Sym_n)$ has yet another $\Z[v,v^{-1}]$-basis---the Murphy basis $\{\fm_{\s\t}\mid \s,\t\in\Std(\lam),\lam\vdash n\}$, which was first introduced by Murphy \cite{Mur}. Both the Murphy basis and the Kazhdan-Lusztig basis $\{C'_w\mid w\in\Sym_n\}$ are cellular bases in the sense of Graham and Lehrer \cite{GL}. The following theorem, which reveals some relationship between the Kazhdan-Lusztig basis $\{C'_w\}$ and the Murphy basis $\{\fm_{\s\t}\}$ of $\HH_v(\Sym_n)$, is actually a (slightly strengthened) reformulation of a remarkable result of Geck.

\begin{thm}[{\cite[Corollaries 4.11 and 5.11]{Gec}}]\label{Geckresult1}
\label{geclPLMS}
Let $\lam\vdash n$. For each $w\in\Sym_n$ with $\RS(w)\in\Std(\lam)\times\Std(\lam)$, we have that $$\begin{aligned}
\fm_{\u_w,\v_w}&= v^{\ell(w_{\lam^*,0})}C'_w+\sum_{\substack{y\in\Sym_n\\ (\u_y,\v_y)\rhd (\u_w,\v_w)}}r_{w,y}C'_y,\\
C'_w&= v^{-\ell(w_{\lam^*,0})}\fm_{\u_w,\v_w}+\sum_{\substack{(\s,\t)\in\Std^2(n)\\ (\s,\t)\rhd (\u_w,\v_w)}}r_{w,\s,\t}\fm_{\s\t},
\end{aligned}
$$
where $r_{w,y}, r_{w,\s,\t}\in\Z[v,v^{-1}]$ for each pair $(w,y)$ and each triple $(w,\s,\t)$.
\end{thm}

\begin{rem}
The formulation in \cite[Corollary 4.11]{Gec} is actually slightly weaker than the theorem we stated above in three points: Firstly, there is a sign $\eta'_w$ in \cite[Corollary 4.11]{Gec} which we removed here; secondly, \cite[Corollary 4.11]{Gec} used the bijection $(\lam,i,j)\mapsto \mathbf{w}_\lam(i,j)$ instead of the correspondence $w\mapsto (\u_w,\v_w)=(P(w)^*,Q(w)^*)$; thirdly, \cite[Corollary 4.11]{Gec} did not
explicitly claim that $\s,\t\in\Std(\lam^*)$ and $r_{w,\s,\t}\neq 0$ imply that $\s\unrhd\u_w$ and $\t\unrhd\v_w$. In fact, $\eta'_w=1$ follows from the proof of \cite[Corollary 5.11]{Gec}. The reconciliation of the two bijection in the second point will be explained in Lemma \ref{recon}, while the third point is implicitly implied  by the proof of \cite[Corollary 4.11]{Gec}.
\end{rem}

The following two theorems are {\bf the second main result} of this paper, which gives a generalization of the above result of Geck.

\begin{thm}\label{mainthm2} Let $\lam\vdash n$. For each $w\in\Sym_n$ with $\RS(w)\in\Std(\lam)\times\Std(\lam)$, we have that $$\begin{aligned}
C'_w &= v^{-\ell(w_{\lam^*,0})}\fm_{\u_w,\v_w}+\sum_{\substack{(\s,\t)\in\Std^2(n)\\ (\s,\t)\Gdom (\u_w,\v_w)}}r_{w,\s,\t}\fm_{\s\t},\\
\fm_{\u_w,\v_w} &= v^{\ell(w_{\lam^*,0})}C'_w+\sum_{\substack{y\in\Sym_n\\ (\u_y,\v_y)\Gdom (\u_w,\v_w)}}r_{w,y}C'_y,
\end{aligned}$$
where $r_{w,\s,\t},r_{w,y}\in\Z[v,v^{-1}]$ for each triple $(w,\s,\t)$ and each pair $(w,y)$.
\end{thm}

Note that $\HH_v(\Sym_n)$ is naturally embedded into $\HH_{\Q(v)}(\Sym_n)$ via $h\mapsto 1\otimes h$. Henceforth, we identify $\HH_v(\Sym_n)$ with its image via this embedding.

\begin{thm}\label{mainthm3} Let $\lam\vdash n$. For each $w\in\Sym_n$ with $\RS(w)\in\Std(\lam)\times\Std(\lam)$, we have that $$\begin{aligned}
C'_w&= v^{-\ell(w_{\lam^*,0})}\mff_{\u_w,\v_w}+\sum_{\substack{(\s,\t)\in\Std^2(n)\\ (\s,\t)\Gdom (\u_w,\v_w)}}r'_{w,\s,\t}\mff_{\s\t},\\
\mff_{\u_w,\v_w}&= v^{\ell(w_{\lam^*,0})}C'_w+\sum_{\substack{y\in\Sym_n\\ (\u_y,\v_y)\Gdom (\u_w,\v_w)}}r'_{w,y}C'_y,
\end{aligned}
$$
where $r'_{w,\s,\t},r'_{w,y}\in\Q(v)$ for each triple $(w,\s,\t)$ and each pair $(w,y)$.
\end{thm}

Kazhdan and Lusztig have also introduced another $\Z[v,v^{-1}]$-basis $\{C_w\mid w\in\Sym_n\}$ which will be called twisted KL basis. Let $\{\mathfrak{g}_{\s\t}\mid \s,\t\in\Std(\lam),\lam\vdash n\}$ be the dual seminormal basis of $\HH_{\Q(v)}(\Sym_n)$ corresponding to the dual Murphy basis $\{\mathfrak{n}_{\s\t}\mid \s,\t\in\Std(\lam),\lam\vdash n\}$ of $\HH_v(\Sym_n)$. In a parallel to the above two theorems, we also obtain the corresponding results (Theorems \ref{mainthm4}, \ref{mainthm5}) for the triangular transition matrices relations between the twisted Kazhdan-Lusztig basis $\{C_w\mid w\in\Sym_n\}$ and the dual seminormal basis $\{\mathfrak{g}_{\s\t}\mid \s,\t\in\Std(\lam),\lam\vdash n\}$ and the dual Murphy basis $\{\mathfrak{n}_{\s\t}\mid \s,\t\in\Std(\lam),\lam\vdash n\}$.

\medskip
The paper is organised as follows. In Section 2 we recall some basic knowledge, notations and definitions on the symmetric groups and its associated Iwahori-Hecke algebra. In particular, we give the definitions of KL basis, twisted KL basis, Murphy basis and seminormal basis. In Section 3 we give the proof of our first main result Theorem \ref{mainthm1} and the second main result Theorems \ref{mainthm2}, \ref{mainthm3} on the relations between the KL basis and Murphy basis and seminormal basis. In Section 4 we first recall the definition of dual Murphy basis and dual seminormal basis of the Hecke algebra $\HH_{\Q(v)}(\Sym_n)$. Then we give the third main result Theorems \ref{mainthm4}, \ref{mainthm5} on the relations between the twisted KL basis and the dual Murphy basis and the dual seminormal basis.

\bigskip\bigskip
\centerline{Acknowledgements}
\bigskip

The research was supported by
the National Natural Science Foundation of China (No. 11901030),
the Natural Science Foundation of Beijing Municipality (No. 1204034)
and
Beijing Institute of Technology Research Fund Program for Young Scholars.
\bigskip

\bigskip
\section{Preliminary}

In this section, we shall give some preliminary definitions and results on the symmetric groups and its associated Iwahori-Hecke algebra.

Let $\Sym_n$ be the symmetric group on $\{1,2,\cdots,n\}$. For each $1\leq i<n$, we set $s_i:=(i,i+1)$. Let $v$ be an indeterminate over $\Z$. Set $\A:=\Z[v,v^{-1}]$. By definition, the Iwahori-Hecke algebra $\HH_v(\Sym_n)$ associated to $\Sym_n$ is the unital associative $\A$-algebra with generators $T_1,\cdots,T_{n-1}$ and defining relations
$$
\begin{aligned}
 (T_r-v)(T_r+v^{-1})&=0,\quad\forall\,1\leq r<n,\\
 T_i T_{i+1} T_i&=T_{i+1} T_i T_{i+1},\quad\forall\,1\leq i<n-1,\\
 T_j T_k&=T_k T_j, \quad\forall\, 1\leq j<k-1\leq n-2 .
\end{aligned}
$$
A word $w=s_{i_{1}}s_{i_{2}}\ldots s_{i_{k}}$ for $w\in \Sym_{n}$ is called a reduced expression of $w$ if $k$ is minimal; in this case we say $w$ has length $k$ and we write $\ell(w)=k$. Given a reduced expression $s_{i_{1}}\cdots s_{i_{k}}$ of $w\in \Sym_{n}$, we define $T_{w}=T_{i_{1}}\cdots T_{i_{k}}$, which is independent of the choice of the reduced expression of $w$ because of the braid relations holds in $\HH_v(\Sym_n)$. For any field $K$ which is an $\A$-algebra, we define $\HH_K(\Sym_n):=K\otimes_{\A}\HH_v(\Sym_n)$.

Let $m$ be a positive integer. A partition of $m$ is a weakly decreasing sequence $\lam=(\lam_{1},\lam_{2},\cdots)$ of non-negative integers such that $|\lam|:=\Sigma_{i\geq 1}\lam_{i}=m$. In this case we write $\lam\vdash m$. We use $\P_m$ to denote the set of partitions of $m$. For any $\lam=(\lam_{1},\lam_{2},\ldots)\in\P_m$, we define $\lam^*=(\lam^*_{1},\lam^*_{2},\ldots)$, where for each $i$, $\lam^*_{i}:=\#\{j\mid \lam_{j} \geq i\}$. Then $\lam^*\in\P_m$ and is called the conjugate of $\lam$. We define the Young diagram of $\lam$ to be $[\lam]:=\{(i,j)\mid 1\leq j\leq \lam_{i}\}$, and each $(i,j)\in [\lam]$ is called a box of $\lam$. A $\lam$-tableau $\t$ is a bijective map $\t\colon [\lam]\to \{1,2,\cdots,m\}$ . We usually call the image of a box $(x, y)\in[\lam]$ under $\t$ the entry of the box $(x,y)$ in $\t$. If $\t$ is a $\lam$-tableau, then we set $\Shape(\t):=\lam$, and define $\t^*\in\Std(\lam^*)$ by $\t^*(i,j):=\t(j,i)$ for any $(i,j)\in[\lam^*]$. We call $\t^*$ the conjugate of $\t$. If the $\lam$-tableau $\t$ satisfies that $\t(i,j)\leq\t(i,b)$ for any $j\leq b$ and any $i$, then we say $\t$ is row standard. A $\lam$-tableau $\t$ is called standard if both $\t$ and $\t^*$ are row standard. We use $\Std(\lam)$ to denote the set of standard $\lam$-tableaux. For each $\lam$-tableau $\t$ and $1\leq k\leq m$, we use $\t\!\downarrow_k$ to denote the subtableau of $\t$ which contains the integer $1,2,\cdots,k$. Set $\Std(m):=\sqcup_{\lam\vdash m}\Std(\lam)$, $\Std^2(m):=\sqcup_{\lam\vdash m}\Std^2(\lam)$, where $\Std^2(\lam):=\Std(\lam)\times\Std(\lam)$ for each $\lam\in\P_m$.

Let $\lam\in\P_n$. Let $\t^\lam$ be the initial standard $\lam$-tableaux with the numbers $1,2,\cdots,n$ entered in order along the rows of $[\lam]$. We define $\t_{\lam}:=(\t^{\lam^*})^*$. In particular, $\t_\lam$ is the standard $\lam$-tableaux in which the numbers $1,2,\cdots,n$ entered in order along the columns of $[\lam]$.

For any $\lam,\mu\in\P_n$, we write $\lam\unrhd\mu$ if for all $i\geq 1$,
$$\sum_{j=1}^{i}\lam_{j}\geq \sum_{j=1}^{i}\mu_{j}.$$
Clearly $\P_n$ is a poset with respect to the partial order ``$\unrhd$''. We call ``$\unrhd$'' the dominance order on $\P_n$. If $\lam\unrhd\mu$ and $\lam\neq\mu$, then we write $\lam\rhd\mu$.

Let $\lam,\mu\in\P_n$. Let $\s\in\Std(\lam), \t\in\Std(\mu)$. We write $\s\unrhd\t$ if for any $1\leq k\leq n$, $${\rm{Shape}}(\s\!\downarrow_k)\unrhd{\rm{Shape}}(\t\!\downarrow_k). $$ If $\s\unrhd\t$ and $\s\neq\t$ then we write $\s\rhd\t$. Clearly, ``$\unrhd$'' is also a partial order which will be called the dominance order on $\Std(n)$, and $\t^\lam\unrhd\s\unrhd\t_\lam$ for any $\s\in\Std(\lam)$.

For each $\lambda=(\lambda_1, \lambda_2,\dots,\lambda_k) \in \P_n$, we define \begin{align*}
I_{\lambda}:=\{s_r \mid 1\leq r<n, r\neq\lambda_1+\lambda_2+\cdots+\lambda_i, \forall\,1\leq i<k\} .
\end{align*}
and denote by $\Sym_\lam$ the standard parabolic subgroup of $\mathfrak{S}_n$ generated by $I_\lam$. Then $\Sym_\lam\cong\mathfrak{S}_{\lambda_1}\times \mathfrak{S}_{\lambda_2}\times \cdots \times \mathfrak{S}_{\lambda_k}$. We allow $\Sym_n$ to act on $\{1,2,\cdots,n\}$ from both the left-hand side and the right-hand side. They are related to each other via $(k)w=w^{-1}(k)$ for any $1\leq k\leq n$ and $w\in\Sym_n$. Similarly, the symmetric group $\Sym_n$ can act on the set of standard tableaux from both the left-hand side and the right-hand side. It is clear that $\Sym_\lam$ is nothing but the row stabilizer of $\t^\lam$ in $\Sym_n$. We set $\mathcal{D}_\lam:=\{d\in\Sym_n\mid \text{$\t^\lam d$ is row standard}\}$. Then $\mathcal{D}_\lam$ is the set of minimal length right coset representatives of $\Sym_\lam$ in $\Sym_n$ (\cite{DJ1}, \cite{DJ2}). In particular, for any $w\in\Sym_n$, there are unique $d\in\mathcal{D}_\lam$ and $z\in\Sym_\lam$ such that $w=zd$ with $\ell(w)=\ell(z)+\ell(d)$. We use $w_{\lam,0}$ to denote the unique longest element in $\Sym_\lam$. For each $\t\in\Std(\lam)$, let $d(\t)\in\Sym_n$ be the unique element in $\Sym_n$ such that $\t^\lam d(\t)=\t$.

Let $\ast$ be the anti-isomorphism of $\HH_v(\Sym_n)$ which is defined on generators by $T_i^\ast=T_i$ for any $1\leq i<n$.

\begin{dfn}[{\cite{Mur}}] Let $\lam\in\P_n$ and $\s,\t\in\Std(\lam)$. We define $$
\fm_{\s\t}:=T_{d(\s)}^\ast\Bigl(\sum_{w\in\Sym_\lam}v^{\ell(w)}T_w\Bigr)T_{d(\t)}.
$$
\end{dfn}

\begin{rem} The readers should identify $d(\s)$ in our notation with $d(\s)^{-1}$ in the notation of \cite{Gec}. In particular, the above definition of $\fm_{\s\t}$ coincides with \cite[Theorem 4.1]{Gec}.
\end{rem}

\begin{thm}[\cite{Mur}] The set $\{\fm_{\s\t}\mid \s,\t\in\Std(\lam),\lam\in\P_n\}$ forms an $\A$-basis of $\HH_{v}(\Sym_n)$.
\end{thm}

We call $\{\fm_{\s\t}\mid \s,\t\in\Std(\lam),\lam\in\P_n\}$ the {\bf Murphy basis} of $\HH_{v}(\Sym_n)$. It is cellular in the sense of Graham and Lehrer \cite{GL}.

Let $a\mapsto\overline{a}$ be the unique ring involution of $\A$ defined by $\overline{v}=v^{-1}$. It extends to a bar involution $h\mapsto\overline{h}$ of $\HH_v(\Sym_n)$ defined by $$
\overline{\sum_{w\in\Sym_n}a_wT_w}:=\sum_{w\in \Sym_n}\overline{a_w}T_{w^{-1}}^{-1} .
$$
Let ``$>$'' be the Bruhat order defined on the symmetric group $\Sym_n$.

\begin{thm}[{\cite{KL}}] For any $w\in W$, there is a unique $C'_w\in\HH_v(\Sym_n)$ such that $\overline{C'_w}=C'_w$ and $$
C'_w=T_w+\sum_{w>y\in\Sym_n}v^{\ell(y)-\ell(w)}P_{y,w}(v^2)T_y ,
$$
where $P_{y,w}(v^2)\in\A$ is a polynomial on $q:=v^2$ of degree $\leq (\ell(w)-\ell(y)-1)/2$ for $y<w$. Moreover, $\{C'_w\mid w\in\Sym_n\}$ forms an $\A$-basis of $\HH_v(\Sym_n)$.
\end{thm}

For any $w\in\Sym_n$, we define $\eps_w:=(-1)^{\ell(w)}$. Let $j$ be the unique ring involution of $\HH_v(\Sym_n)$ given by $j\bigl(\sum_{w\in \Sym_n}a_wT_w\bigr)=\sum_{w\in\Sym_n}\overline{a_w}\eps_wT_w$. For any $w\in\Sym_n$, we define $C_w:=\eps_wj(C'_w)$. Then we have that $\overline{C_w}=C_w$, and $$
C_w=T_w+\sum_{w>y\in\Sym_n}\eps_y\eps_wv^{\ell(w)-\ell(y)}P_{y,w}(v^{-2})T_y .
$$
Moreover, $\{C_w\mid w\in\Sym_n\}$ forms an $\A$-basis of $\HH_v(\Sym_n)$. We call $\{C'_w\mid w\in\Sym_n\}$ the Kazhdan-Lusztig basis (or KL basis for short) of $\HH_v(\Sym_n)$, and $\{C_w\mid w\in\Sym_n\}$ the twisted Kazhdan-Lusztig basis (or twisted KL basis for short) of $\HH_v(\Sym_n)$.

\begin{lem}[{\cite{KL}}]\label{mucoeff} Let $w\in\Sym_n$. Then for any $1\leq k<n$, we have that $$\begin{aligned}
C'_{s_k}C'_w=\begin{cases} C'_{s_kw}+\sum_{\substack{z\in\Sym_n\\ sz<z<w}}\mu(z,w)C'_z, &\text{if $\ell(s_kw)=\ell(w)+1$;}\\
(v+v^{-1})C'_w, &\text{if $\ell(s_kw)=\ell(w)-1$.}
\end{cases}
\end{aligned}
$$
where $\mu(z,w)$ is the leading coefficient defined in \cite[Definition 1.2]{KL}.
\end{lem}

Let $\Q(v)$ be the rational functional field on $v$. It is well-known that the Hecke algebra $\HH_{\Q(v)}(\Sym_n)$ is split semisimple. For any $\lam\in\P_n$, $\t\in\Std(\lam)$ and any $1\leq k\leq n$, we define \begin{equation}
c_{\t}(k)=v^{2(j-i)},\quad  \text{if $k$ appears in row $i$ and column $j$ of $\t$.}
\end{equation}
We define $C(k):=\{c_\t(k) \mid  \t\in\Std(\lam), \lam\in\P_n\}$. In particular, for any $\lam,\mu\in\P_n$, $\s\in\Std(\lam), \t\in\Std(\mu)$, if $\s\neq\t$, then there exists $1\leq k\leq n$ such that $c_\s(k)-c_\t(k)\in \Q(v)^\times$, where $\Q(v)^\times$ denotes the set of invertible elements in $\Q(v)$.

For each integer $1\leq k\leq n$, we define $$
L_k:=T_{k-1}\cdots T_2T_1^2T_2\cdots T_{k-1} .
$$

\begin{dfn}[{\cite[Definition 2.4]{Ma}}]
\label{Ft}
Let $\lam\in\P_n$ and $\t\in\Std(\lam)$. We define
$$F_{\t}=\prod\limits^n\limits_{k=1}\prod\limits_{\substack{c\in C(k)\\c\neq c_{\t}(k)}}\frac{L_k-c}{c_{\t}(k)-c}.$$
\end{dfn}

For any $\lam\in\P_n$ and $\s,\t\in\Std(\lam)$, we define \begin{equation}\label{seminormal}
\mff_{\s\t}:=F_\s \fm_{\s\t}F_\t .\end{equation}

\begin{lem}[{\cite[Corollary 2.7, Theorems 2.11, Corollary 2.13]{Ma}}]
\label{obobn}
~\\
\begin{enumerate}
\item [1)] The set $\{\mff_{\s\t}\mid \s,\t \in \Std(\lam), \lam\in\P_n\}$ is a $\Q(v)$-basis of $\HH_{\Q(v)}(\Sym_n)$;
\item[2)] If $\s, \t, \u$ and $\v$ are standard tableaux, then $\mff_{\s\t}\mff_{\u\v}=\delta_{\t\u}\gamma_{\t}\mff_{\s\v}$, where $\gamma_\t\in\Q(v)^\times$;
\item[3)] For each $\lam\in\P_n$, we define $S(\lam):=f_{\t^\lam\t^\lam}\HH_{\Q(v)}(\Sym_n)$. Then $S(\lam)\cong\mff_{\s\t}\HH_{\Q(v)}(\Sym_n)$ for any $\s,\t\in\Std(\lam)$, and $\{S(\lam)\mid \lam\vdash n\}$ is a complete set of pairwise non-isomorphic simple right $\HH_{\Q(v)}(\Sym_n)$-modules; Similarly, we define $S'(\lam):=\HH_{\Q(v)}(\Sym_n)f_{\t^\lam\t^\lam}$. Then $S'(\lam)\cong\HH_{\Q(v)}(\Sym_n)\mff_{\s\t}$ for any $\s,\t\in\Std(\lam)$, and $\{S'(\lam)\mid \lam\vdash n\}$ is a complete set of pairwise non-isomorphic simple left $\HH_{\Q(v)}(\Sym_n)$-modules;
\item[4)] Let $\lam\in\P_n$ and $\s,\u\in\Std(\lam)$. For any integer $1\leq i<n$, if $\t:=\s(i,i+1)$ is standard then
	$$\mff_{\u\s}T_i=\begin{cases*}
		\frac{(v-v^{-1})c_\s(i+1)}{c_\s(i+1)-c_\s(i)}\mff_{\u\s}+\mff_{\u\t}, &\text{if $\t\triangleleft\s$},\\
\frac{(v-v^{-1})c_\s(i+1)}{c_\s(i+1)-c_\s(i)}\mff_{\u\s}+\frac{(vc_{\s}(i)-v^{-1}c_{\s}(i+1))(v^{-1}c_{\s}(i)-vc_{\s}(i+1))}
{(c_{\s}(i+1)-c_{\s}(i))^2}\mff_{\u\t}, &\text{if $\s\triangleleft\t$},
	\end{cases*}$$
If $\t:=\s(i,i+1)$ is not standard then
	$$\mff_{\u\s}T_i=\begin{cases*}
		v\mff_{\u\s}, &\text{if $i$ and $i+1$ are in the same row of $\s$},\\
		-v^{-1}\mff_{\u\s}, &\text{if $i$ and $i+1$ are in the same column of $\s$}.
	\end{cases*}$$
\end{enumerate}
Similar formulae hold for $T_i\mff_{\u\s}$.
\end{lem}

We call $\{\mff_{\s\t}\mid \s,\t \in \Std(\lam), \lam\in\P_n\}$ the {\bf seminormal basis} of $\HH_{\Q(v)}(\Sym_n)$ corresponding to the Murphy basis $\{\fm_{\s\t}\mid \s,\t \in \Std(\lam), \lam\in\P_n\}$ of $\HH_{v}(\Sym_n)$. Actually the seminormal bases can be defined for the more general cyclotomic Hecke algebras of type $G(\ell,1,n)$ and play an important role in many aspects of the representation theory of Hecke algebras and KLR algebras, see \cite{Ma}, \cite{HuMathas:GradedCellular}, \cite{HuMathas:SeminormalQuiver}. The formulae in Lemma \ref{obobn} 4) gives rise to a {\bf seminormal basis} $\{\mff_\s \mid \s\in\Std(\lam)\}$ of the simple right $\HH_{\Q(v)}(\Sym_n)$-module $S(\lam)$ and the explicit matrix form of each Hecke generator $T_i$ on this basis. We usually call it {\bf the seminormal form} of $S(\lam)$ (\cite{Hoe}, \cite{Y}). Similar result holds for the simple left $\HH_{\Q(v)}(\Sym_n)$-module $S'(\lam)$.

\begin{lem}[{\cite{HuMathas:Graded Induction}}]
\label{hm2}
For any $\lam\in\P_n$ and $\s,\t\in\Std(\lam)$, we have $$\begin{aligned}
\fm_{\s\t}&=\mff_{\s\t}+\sum_{\substack{(\u,\v)\in\Std^2(n)\\ (\u,\v)\Gdom(\s,\t)}}r^{\s\t}_{\u\v}\mff_{\u\v},\\
\mff_{\s\t}&=\fm_{\s\t}+\sum_{\substack{(\u,\v)\in\Std^2(n)\\ (\u,\v)\Gdom(\s,\t)}}{\hat{r}}^{\s\t}_{\u\v}\fm_{\u\v},
\end{aligned} $$
where $r^{\s\t}_{\u\v}, {\hat{r}}^{\s\t}_{\u\v}\in\Q(v)$ for each pair $(\u,\v)$.
\end{lem}

\bigskip
\section{Kazhdan-Lusztig bases and seminormal bases}

The purpose of this section is to give the proof of Theorems \ref{mainthm2} and \ref{mainthm3} for the triangular transition matrices relations between the KL basis $\{C'_w\mid w\in\Sym_n\}$, the seminormal basis $\{\mathfrak{f}_{\s\t}\mid \s,\t\in\Std(\lam),\lam\vdash n\}$ and the Murphy basis $\{\mathfrak{m}_{\s\t}\mid \s,\t\in\Std(\lam),\lam\vdash n\}$.

The following definitions of Kazhdan-Lusztig preorders $\leq_L$, $\leq_R$ and $\leq_{LR}$ are equivalent to their original definitions in \cite{KL}.

\begin{dfn} Let $x,y\in\Sym_n$. If there exists some $h\in\HH_v(\Sym_n)$ such that $C'_x$ appears with a nonzero coefficient in the expansion of $hC'_y$ as an $\A$-linear combination of the $C'$-bases, then we define $x\leq_L y$. Similarly, if there exists some $h'\in\HH_v(\Sym_n)$ such that $C'_x$ appears with a nonzero coefficient in the expansion of $C'_yh'$ as an $\A$-linear combination of the $C'$-bases, then we define $x\leq_R y$. We use $\leq_{LR}$ to denote the preorder generated by the two preorders $\leq_L$ and $\leq_R$. We call $\leq_L, \leq_R$ and $\leq_{LR}$ the (Kazhdan-Lusztig) left cell preorder, right cell preorder and two-sided cell preorder respectively.
\end{dfn}

It is well-known that $x\leq_L y$ if and only if $x^{-1}\leq_R y^{-1}$. The equivalence classes generated by $\leq_L$, $\leq_R$ and $\leq_{LR}$ will be called the (Kazhdan-Lusztig) left cells, right cells and two-sided cells of $\Sym_n$, which will be denoted by ``$\sim_L$'', ``$\sim_R$'' and ``$\sim_{LR}$'' respectively.

Robinson-Schensted-Knuth correspondence gives a combinatorial modes to describe the (Kazhdan-Lusztig) left cells, right cells and two-sided cells of $\Sym_n$. To state its definition, we first recall the following row-insertion algorithm.

\begin{dfn} Let $\lam\in\P_m$ and $\t$ be a standard $\lam$-tableau with entries in $\N$. Let $i\in\N$ such that none of the entry of $\t$ is equal to $i$. We define a new standard tableau $\t\leftarrow i$ of size $m+1$ inductively as follows: \begin{enumerate}
\item If $i$ is larger than all the entries in the first row of $\t$, then $\t\leftarrow i$ is just obtained by adding a box with entry $i$ at the end of the first row of $\t$;
\item Otherwise, let $\t'$ be the tableau obtained from $\t$ by removing the first row and let $i'$ be the left-most entry of the first row of $\t$ such that $i'>i$. Then this entry $i'$ is replaced by $i$ (which gives a new row $R$), and we obtain $\t\leftarrow i$ by adding as a first row $R$ to the standard tableau $\t'\leftarrow i'$.
\end{enumerate}
\end{dfn}

\begin{dfn} For any $w\in\Sym_n$, we define $$
P^{[n]}(w):=(\cdots((\emptyset\leftarrow w(1))\leftarrow w(2))\leftarrow\cdots \leftarrow w(n-1))\leftarrow w(n),\quad Q^{[n]}(w):=P^{[n]}(w^{-1}).
$$
\end{dfn}
If the context is clear, we shall omit the superscript ``$[n]$'' and write $P(w), Q(w)$ instead of $P^{[n]}(w), Q^{[n]}(w)$.

\begin{example}
  If $w =s_1s_3s_4\in \mathfrak{S}_6$, then we have
  \begin{align*}
    P(w)=P^{[6]}(w) = \begin{ytableau}
            1 & 3 & 5 & 6\\
            2 & 4\\
            \end{ytableau}~,~~
    Q(w)=Q^{[6]}(w) = \begin{ytableau}
            1 & 3 & 4 & 6\\
            2 & 5\\
            \end{ytableau}~.\\
    \end{align*}
\end{example}

\begin{thm}[Robinson-Schensted-Knuth correspondence {\cite[\S4.1]{F}}] The map \begin{align*}
\renewcommand\arraystretch{1.3}
\begin{array}{cccc}
\mathrm{RSK}\colon
& \mathfrak{S}_n
& \longrightarrow
& \Std^2(n)\\
& w
& \longmapsto
& ({P}^{[n]}(w), {Q}^{[n]}(w))
\end{array}
\end{align*}
is bijective.
\end{thm}

If $\s,\t\in\Std(\lam)$, where $\lam\in\P_n$, then we define \begin{equation}\label{inverse}
\pi_\lam(\s,\t):=\RS^{-1}(\s,\t)\in\Sym_n .
\end{equation}

In order to simplify the notations, we set \begin{equation}\label{dmn}
\dmn:=\mathcal{D}_{(m,1^{n-m})},
\end{equation}
which is the set of minimal length right coset representatives of $\Sym_m=\Sym_{(m,1^{n-m})}$ in $\Sym_n$. Then $\dmn^{-1}:=\{d^{-1}\mid d\in\dmn\}$ is the set of minimal length left coset representatives of $\Sym_m$ in $\Sym_n$.

\begin{lem}\label{wTow1-n-1}
Let $w\in \mathfrak{S}_{n}$. There exists a unique pair $(d_{n-1},w_{n-1})\in \mathcal{D}_{n-1}(n)^{-1}\times \mathfrak{S}_{(n-1,1)}$ such that $w=d_{n-1}w_{n-1}$.
Moreover, ${Q}^{[n-1]}(w_{n-1})$ is obtained from ${Q}(w)$ by removing the box containing $n$,
i.e. \begin{align*}
Q^{[n]}(w)\downarrow_{n-1}={Q}^{[n-1]}(w_{n-1}).
\end{align*}
\end{lem}

\begin{proof} Since $Q(w)=P(w^{-1})$, it suffices to show that $P^{[n]}(w^{-1})\downarrow_{n-1}=P^{[n-1]}(w_{n-1}^{-1})$. We set $\lam:=(n-1,1)$. There exists a unique $1\leq k\leq n$ such that $(n)d_{n-1}^{-1}=k=d_{n-1}(n)$.
Since $\t^\lam d_{n-1}^{-1}$ is row standard, it follows that \begin{equation}\label{dn1}
d_{n-1}(j)=(j)d_{n-1}^{-1}=\begin{cases} j, &\text{if $1\leq j<k$,}\\
j+1, &\text{if $k\leq j\leq n-1$,}\\
k, &\text{if $j=n$.}
\end{cases}
\end{equation}
It follows that  $$
d_{n-1}^{-1}(j)=\begin{cases} j, &\text{if $1\leq j<k$,}\\
j-1, &\text{if $k+1\leq j\leq n$,}\\
n, &\text{if $j=k$.}
\end{cases}
$$
Now assume that $i_j=w_{n-1}^{-1}(j)$ for each $1\leq j\leq n-1$. Then $(i_1,\cdots,i_{n-1})$ is a permutation of $(1,2,\cdots,n-1)$ because $w_{n-1}^{-1}\in\Sym_{n-1}$. It follows directly from the definition that
$$
P^{[n-1]}(w_{n-1}^{-1})=P^{[n]}(w_{n-1}^{-1}d_{n-1}^{-1})\downarrow_{n-1}.
$$
This completes the proof of the lemma.
\end{proof}

\begin{prop}\label{prop:wTow1-n-m} Let $w\in \mathfrak{S}_{n}$ and $m\leq n$. There exists a unique pair $(d_{m},w_{m})\in \dmn^{-1}\times \mathfrak{S}_{m}$ such that $w=d_{m}w_{m}$.
Moreover, ${Q}^{[m]}(w_m)$ is obtained from ${Q}^{[n]}(w)$ by removing the boxes containing $m+1, m+2, \dots, n$,
i.e., \begin{align*}
{P}^{[n]}(w^{-1})\downarrow_m={Q}^{[n]}(w)\downarrow_m={Q}^{[m]}(w_m)={P}^{[m]}(w_m^{-1}) .
\end{align*}
\end{prop}

\begin{proof} We use induction on $n-m$. If $n-m=0$, then the proposition holds trivially. If $n-m=1$, then the proposition follows from Lemma \ref{wTow1-n-1}.
Now let $k\geq 2$ and suppose that the proposition holds whenever $n-m<k$. We consider the case when $k=n-m\geq 2$.

By Lemma \ref{wTow1-n-1},
there exists a unique pair $(d_{n-1},w_{n-1})\in \mathcal{D}_{n-1}(n)^{-1}\times \mathfrak{S}_{(n-1,1)}$ such that $w=d_{n-1}w_{n-1}$ and
\begin{equation}\label{nn1}
{Q}^{[n]}(w)\downarrow_{n-1}={Q}^{[n-1]}(w_{n-1}).
\end{equation}

Now we have that $(n-1)-m<n-m=k$. By induction hypothesis, there exists a unique pair $(d'_m,{w}_m)\in \mathcal{D}_m(n-1)^{-1}\times \mathfrak{S}_{(m,1^{n-1-m})}$
such that ${w}_{n-1}={d}'_m {w}_m$
and
\begin{align*}
{Q}^{[n-1]}({w}_{n-1})\downarrow_m={Q}^{[m]}({w}_m) .
\end{align*}

Now we define $d_m:=d_{n-1}{d}'_m$. It is clear that $d_m\in\dmn^{-1}$ because $d_{n-1}{d}'_m\t^{(m,1^{n-m})}$ is row standard. Then
$$w=d_{n-1} w_{n-1}
  =d_{n-1} d'_m  w_m
  =d_m w_m $$
and by (\ref{nn1}),
$$
{Q}^{[n]}(w)\downarrow_m
= {Q}^{[n-1]}(w_{n-1})\downarrow_m
= {Q}^{[m]}({w}_m) .
$$
Therefore, the proposition is proved.
\end{proof}

\begin{example}
  Let $n=6,~w =s_1s_3s_4\in \mathfrak{S}_6$ and $m=3$. Then $d_{3}=s_{3}s_{4}\in\dmn^{-1}$ and $w_{3}=s_{1}\in\mathfrak{S}_3$, $w=d_3w_{3}$. We have
  \begin{align*}
    Q(s_1s_3s_4) & = \begin{ytableau}
            1 & 3 & 4 & 6\\
            2 & 5\\
            \end{ytableau}~,\\
    Q(s_1s_3s_4)\downarrow_3 = Q(s_1) & = \begin{ytableau}
            1 & 3\\
            2\\
            \end{ytableau}~.\\
  \end{align*}
\end{example}

Let $W_I$ be a standard parabolic subgroup of $\Sym_n$. Applying \cite[Corollary 8.2.4]{B}, we can deduce that for any $x,y\in W_I$, $x\leq_L y$ (resp., $x\leq_R y$) with respect to $W_I$ if and only if $x\leq_L y$ (resp., $x\leq_R y$) with respect to $\Sym_n$. The following useful result of Geck will play
an important role in the proof of our first main result Theorem \ref{mainthm1}.

\begin{lem}[{\cite[Proposition 3.3]{Gec03}}]
 \label{dmu1}
 Let $w=dz\in\Sym_n$, where $d\in\dmn^{-1}$, $z\in\Sym_m$, $1\leq m\leq n$. Then we have that $$\begin{aligned}
C'_{w}&=T_dC'_z+\sum_{\substack{u\in\dmn^{-1}, y\in\Sym_m\\ u<d, y\leq_L z}}\widetilde{r}_{w,u,y}T_uC'_y,\\
C'_{w^{-1}}&=C'_{z^{-1}}T_{d^{-1}}+\sum_{\substack{u\in\dmn, y\in\Sym_m\\ u<d^{-1}, y\leq_R z^{-1}}}\widetilde{r}_{w,u,y}C'_yT_u,
\end{aligned}
$$
where $\widetilde{r}_{w,u,y}\in\A$ for each triple $(w,u,y)$.
\end{lem}

\begin{proof} The first equality follows from \cite[Proposition 3.3]{Gec03} and the sentence in the paragraph above this lemma. The second equality follows from the first one by applying the anti-isomorphism ``$\ast$'' of $\HH_v(\Sym_n)$.
\end{proof}

Let $\lam\in\P_n$. We assume that $\mathcal{D}_\lam^{-1}=\{x_i \mid 1\leq i\leq d_\lam\}$, where $d_\lam:=\dim S(\lam)$. Let $(\lam,i,j)\mapsto \mathbf{w}_\lam(i,j)$ be the bijection defined in \cite[Lemma 3.4]{Gec}. Recall that $w_{\lam,0}$ is the unique longest element in $\Sym_\lam$.

\begin{lem}\label{recon} Let $\lam\in\P_n$ and $\s,\t\in\Std(\lam)$. Let $1\leq i,j\leq d_\lam$ be such that $x_i\t^\lam=\s, x_j\t^\lam=\t$. Then we have that $$
P(d(\s)^{-1}w_{\lam,0})=\s^*, \quad Q(w_{\lam,0}d(\t))=\t^*,\quad \mathbf{w}_\lam(i,j)=\pi_{\lam^*}(\s^*,\t^*) .
$$
\end{lem}

\begin{proof} The first equality follows directly from the definition of the RSK correspondence and verification. The second equality follows from the first one because $Q(w)=P(w^{-1})$.

By the second row in  \cite[Page 658]{Gec}, we have $\mathbf{w}_\lam(i,1)=d(\s)^{-1}w_{\lam,0}$. Applying \cite[Remark 3.5(a]{Gec}, we can deduce that $\mathbf{w}_\lam(1,j)=w_{\lam,0}d(\t)$.
Combining this with the first two equalities we can conclude that $\mathbf{w}_\lam(i,1)$ lies in the Kazhdan-Lusztig right cell determined by $\s^*$ (under the RSK correspondence), while $\mathbf{w}_\lam(1,j)$ lies in the Kazhdan-Lusztig left cell determined by $\t^*$ (under the RSK correspondence).

On the other hand, by the proof of \cite[Theorem 5.4]{Gec}, we see that $\mathbf{w}_\lam(i,j)$ is the unique element in the intersection of the Kazhdan-Lusztig right cell
which contains $\mathbf{w}_\lam(i,1)$  with the Kazhdan-Lusztig left cell which contains $\mathbf{w}_\lam(1,j)$. It follows that $\mathbf{w}_\lam(i,j)=\pi_{\lam^*}(\s^*,\t^*)$. This completes the proof of the lemma.
\end{proof}

\medskip
\noindent
{\bf Proof of Theorem \ref{Geckresult1}}: This follows from \cite[Corollaries 4.11 and 5.11]{Gec} and Lemma \ref{recon}.\qed
\medskip

\begin{lem}\label{resm} Let $\lam\in\P_n$ and $\t\in\Std(\lam)$. Then for any $0\neq h=\sum_{\s\in\Std(\lam)}r_{\s}\mff_{\s\t}\in\HH_{\Q(v)}(\Sym_n)$ and
any $1\leq m\leq n$, we have $S(\mu)\cong h\HH_{\Q(v)}(\Sym_m)$, where
$\mu:=\Shape\bigl(\t\downarrow_{m}\bigr)$ and $r_{\s}\in\Q(v)$ for each pair $(\s,\t)$. In particular, $S(\mu)\cong\mff_{\s\t}\HH_{\Q(v)}(\Sym_m)$ for any $\s\in\Std(\lam)$.
\end{lem}

\begin{proof} Let $1\leq m\leq n$. We define $\mu:=\Shape\bigl(\t\downarrow_{m}\bigr)$ and $$
\{\t_1,\cdots,\t_k\}:=\{\t w\mid \text{$\t w\in\Std(\lam)$ for some $w\in\Sym_m$}\} .
$$
By Lemma \ref{obobn}, it is easy to see that $\{h_i:=\sum_{\s\in\Std(\lam)}r_{\s}\mff_{\s\t_i}\mid 1\leq i\leq k\}$ forms a $\Q(v)$-basis of $h\HH_{\Q(v)}(\Sym_m)$.

For each $1\leq i\leq m$, comparing the seminormal form of $S(\mu)$ and the right action of $T_i$ on the $\Q(v)$-basis $\{h_i\mid 1\leq i\leq k\}$ of
$h\HH_{\Q(v)}(\Sym_m)$ given in Lemma \ref{obobn}, we can deduce that there is a natural surjective homomorphism $\pi: S(\mu)\twoheadrightarrow h\HH_{\Q(v)}(\Sym_m)$. Since $S(\mu)$ is a simple $\HH_{\Q(v)}(\Sym_m)$-module, $\pi$ has to be an isomorphism. This proves that $h\HH_{\Q(v)}(\Sym_m)\cong S(\mu)$.
\end{proof}

\begin{lem}\label{Cseminormal} Let $\lam\vdash n$. For each $w\in\Sym_n$ with $\RS(w)\in\Std(\lam)\times\Std(\lam)$, we have that $$\begin{aligned}
C'_w&= v^{-\ell(w_{\lam^*,0})}\mff_{\u_w,\v_w}+\sum_{\substack{(\s,\t)\in\Std^2(n)\\ (\s,\t)\rhd (\u_w,\v_w)}}r'_{w,\s,\t}\mff_{\s\t},\\
\mff_{\u_w,\v_w}&= v^{\ell(w_{\lam^*,0})}C'_w+\sum_{\substack{y\in\Sym_n\\ (\u_y,\v_y)\rhd (\u_w,\v_w)}}r'_{w,y}C'_y,
\end{aligned}$$
where $r'_{w,\s,\t},r'_{w,y}\in\Q(v)$ for each triple $(w,\s,\t)$ and each pair $(w,y)$.
\end{lem}

\begin{proof} This follows from a combination of Lemma \ref{geclPLMS} and Lemma \ref{hm2}. Note that due to the difference between ``$\Gdom$'' and ``$\rhd$'', the statement of this lemma is only a weaker version of Theorem \ref{mainthm3}.
\end{proof}

\begin{cor}\label{keycor1} Let $\lam\vdash n$. For each $w\in\Sym_n$ with $\RS(w)\in\Std(\lam)\times\Std(\lam)$, there exists an integer $r>0$ such that as a right $\HH_{\Q(v)}(\Sym_n)$-module, $C'_{w}\HH_{\Q(v)}(\Sym_n)$ is a direct summand of $$
\Bigl(\bigoplus_{\lam^*\unlhd\mu\vdash n} S(\mu)\Bigr)^{\oplus r}.
$$
Similarly, there exists an integer $r>0$ such that as a left $\HH_{\Q(v)}(\Sym_n)$-module, $\HH_{\Q(v)}(\Sym_n)C'_{w}$ is a direct summand of $$
\Bigl(\bigoplus_{\lam^*\unlhd\mu\vdash n} S'(\mu)\Bigr)^{\oplus r}.
$$
\end{cor}

\begin{proof} We only prove the first half of the corollary as the second half can be proved in a similar argument. Applying Lemma \ref{Cseminormal}, we get that $$
C'_{w}\HH_{\Q(v)}(\Sym_n)=\Bigl( v^{-\ell(w_{\lam^*,0})}\mff_{\u_w,\v_w}+\sum_{\substack{(\s,\t)\in\Std^2(n)\\ (\s,\t)\rhd (\u_w,\v_w)}}r'_{w,\s,\t}\mff_{\s\t}\Bigr)\HH_{\Q(v)}(\Sym_n)
$$
is contained in $$
M_1:=\mff_{\u_w,\v_w}\HH_{\Q(v)}(\Sym_n)+\sum_{\substack{(\s,\t)\in\Std^2(n)\\ (\s,\t)\rhd (\u_w,\v_w)}}\mff_{\s\t}\HH_{\Q(v)}(\Sym_n).
$$
Since there is a natural surjection from $$
M_2:=\mff_{\u_w,\v_w}\HH_{\Q(v)}(\Sym_n)\bigoplus\bigoplus_{\substack{(\s,\t)\in\Std^2(n)\\ (\s,\t)\rhd (\u_w,\v_w)}}\mff_{\s\t}\HH_{\Q(v)}(\Sym_n)
$$
onto $M_1$ and $\HH_{\Q(v)}(\Sym_n)$ is semisimple, and $$
\mff_{\u_w,\v_w}\HH_{\Q(v)}(\Sym_n)\cong S(\lam^*),\quad \mff_{\s\t}\HH_{\Q(v)}(\Sym_n)\cong S(\Shape(\t)),
$$
it follows that $S(\mu)$ is a direct summand of $C'_{w}\HH_{\Q(v)}(\Sym_n)$ only if $\mu\unrhd\lam^*$, from which the first half of the lemma follows.
\end{proof}

%

\medskip
\noindent
{\bf Proof of Theorem \ref{mainthm3}}: It suffices to prove the first equality in Theorem \ref{mainthm3} holds for $C'_w$. We use induction on $\ell(w)$. If $\ell(w)=0$ then $w=1$ and hence $\u_w=\v_w=\t^{(1^n)}$. Since $\s\gdom \t^{(1^n)}$ for any $\s\in\Std(n)$, the first equality in Theorem \ref{mainthm3} now follows from Lemma \ref{Cseminormal}.

Let $w\in\Sym_n$ with $\ell(w)=k>0$. Assume the first equality in Theorem \ref{mainthm3} holds for any $C'_{\hat{w}}$ with $\ell(\hat{w})<k$. We now consider $C'_w$. By Lemma \ref{Cseminormal}, we have that $$
C'_w= v^{-\ell(w_{\lam^*,0})}\mff_{\u_w,\v_w}+\sum_{\substack{(\s,\t)\in\Std^2(n)\\ (\s,\t)\rhd (\u_w,\v_w)}}r'_{w,\s,\t}\mff_{\s\t},
$$
where $r'_{w,\s,\t}\in\Q(v)$ for each triple $(w,\s,\t)$.

Suppose that the first equality in Theorem \ref{mainthm3} does not hold for $C'_w$. In other words, there exists some $(\s,\t)\in\Std^2(n)$ such that  $(\s,\t)\rhd (\u_w,\v_w)$, $r'_{w,\s,\t}\neq 0$, but $(\s,\t)\not\Gdom (\u_w,\v_w)$. We fix such a pair $(\s,\t)$. Then by definition of ``$\Gdom$'', we can find an integer $1\leq m<n$ such that
either $\Shape(\s\downarrow_m)\ntrianglerighteq\Shape(\u_w\downarrow_m)$ or $\Shape(\t\downarrow_m)\ntrianglerighteq\Shape(\v_w\downarrow_m)$.

Suppose that $\Shape(\t\downarrow_m)\ntrianglerighteq\Shape(\v_w\downarrow_m)$. Then applying Lemma \ref{obobn} and Lemma
\ref{resm}, we can deduce that \begin{equation}\label{sta1}
\begin{matrix}\text{there exists some $\mu\in\P_m$ with $\mu\ntrianglerighteq\Shape(\v_w\downarrow_m)$, such that $S(\mu)$ occurs}\\
\text{as a direct summand of $C'_w\HH_{\Q(v)}(\Sym_m)$.}\end{matrix}
\end{equation}

On the other hand, we can decompose $w=dz$, where $d\in\dmn^{-1}$, $z\in\Sym_m$, where $1\leq m< n$. Applying Lemma \ref{dmu1}, we can deduce that $$
C'_{w}\HH_{\Q(v)}(\Sym_m)=\Bigl(T_{d}C'_z+\sum_{\substack{u\in\dmn^{-1}, y\in\Sym_m\\ u<d, y\leq_L z}}\widetilde{r}_{w,u,y}T_uC'_y\Bigr)\HH_{\Q(v)}(\Sym_m)
$$
is contained in $$
M_1(d,z):=T_dC'_z\HH_{\Q(v)}(\Sym_m)+\sum_{\substack{u\in\dmn^{-1}, y\in\Sym_m\\ u<d, y\leq_L z}}\widetilde{r}_{w,u,y}T_uC'_y\HH_{\Q(v)}(\Sym_m) .
$$
There is a natural surjection from $$
M_2(d,z):=C'_z\HH_{\Q(v)}(\Sym_m)\bigoplus\Bigl(\bigoplus_{\substack{u\in\dmn^{-1}, y\in\Sym_m\\ u<d, y\leq_L z}}\widetilde{r}_{w,u,y}C'_y\HH_{\Q(v)}(\Sym_m)\Bigr).
$$
onto $M_1(d,z)$. Since $\HH_{\Q(v)}(\Sym_m)$ is semisimple, it follows that $M_1(d,z)$ is a direct summand of $M_2(d,z)$. Hence $C'_{dz}\HH_{\Q(v)}(\Sym_m)$ is a direct summand of $M_2(d,z)$.

Applying Corollary \ref{keycor1}, we can deduce that for any $\mu\in\P_m$, $S(\mu)$ is a direct summand  $C'_z\HH_{\Q(v)}(\Sym_m)$ only if $\mu\unrhd\Shape(\v_z\downarrow_m)=\Shape(Q^{[m]}(z))^*$. Similarly, by Corollary \ref{keycor1}, for any $y\in\Sym_m$ with $y\leq_L z$, $S(\mu)$ is a direct summand  $C'_y\HH_{\Q(v)}(\Sym_m)$ only if $\mu\unrhd\Shape(\v_y\downarrow_m)\unrhd\Shape(\v_z\downarrow_m)$. However, applying Proposition \ref{prop:wTow1-n-m}, we can deduce that $$\begin{aligned}
\Shape(\v_z\downarrow_m)&=\Shape(Q^{[m]}(z))^*=\Shape(Q^{[n]}(dz)\downarrow_{m})^*=\Shape(Q^{[n]}(w)\downarrow_{m})^*\\
&=\Shape(Q^{[n]}(w)^*\downarrow_{m})=\Shape(\v_w\downarrow_{m}) .
\end{aligned}
$$

It follows that there exists an integer $r>0$ such that as a right $\HH_{\Q(v)}(\Sym_m)$-module, $C'_{w}\HH_{\Q(v)}(\Sym_m)$ is a direct summand of $$
\Bigl(\bigoplus_{\substack{\mu\vdash m\\ \mu\unrhd\Shape(\v_{w}\downarrow_{m})}} S(\mu)\Bigr)^{\oplus r},
$$
which is a contradiction to (\ref{sta1}).

Finally, suppose that $\Shape(\s\downarrow_m)\ntrianglerighteq\Shape(\u_w\downarrow_m)$. Then applying Lemma \ref{obobn} and Lemma
\ref{resm}, we can deduce that \begin{equation}\label{sta2}
\begin{matrix}\text{there exists some $\mu\in\P_m$ with $\mu\ntrianglerighteq\Shape(\u_w\downarrow_m)$, such that $S'(\mu)$ occurs}\\
\text{as a direct summand of $\HH_{\Q(v)}(\Sym_m)C'_w$.}\end{matrix}
\end{equation}

We can decompose $w=zd$, where $d\in\dmn$, $z\in\Sym_m$, where $1\leq m< n$. Applying Lemma \ref{dmu1}, we can deduce that $$
\HH_{\Q(v)}(\Sym_m)C'_{w}=\HH_{\Q(v)}(\Sym_m)\Bigl(C'_zT_{d}+\sum_{\substack{u\in\dmn, y\in\Sym_m\\ u<d, y\leq_R z}}p_{w,u,y}C'_yT_u\Bigr)
$$
is contained in $$
M'_1(d,z):=\HH_{\Q(v)}(\Sym_m)C'_zT_d+\sum_{\substack{u\in\dmn, y\in\Sym_m\\ u<d, y\leq_R z}}p_{w,u,y}\HH_{\Q(v)}(\Sym_m)C'_yT_u ,
$$
where $p_{w,u,y}\in\A$ for each triple $(w,u,y)$. There is a natural surjection from $$
M'_2(d,z):=\HH_{\Q(v)}(\Sym_m)C'_z\bigoplus\Bigl(\bigoplus_{\substack{u\in\dmn, y\in\Sym_m\\ u<d, y\leq_R z}}p_{w,u,y}\HH_{\Q(v)}(\Sym_m)C'_y\Bigr).
$$
onto $M'_1(d,z)$. Since $\HH_{\Q(v)}(\Sym_m)$ is semisimple, it follows that $M'_1(d,z)$ is a direct summand of $M'_2(d,z)$. Hence $\HH_{\Q(v)}(\Sym_m)C'_{zd}$ is a direct summand of $M'_2(d,z)$.

Applying Corollary \ref{keycor1}, we can deduce that for any $\mu\in\P_m$, $S'(\mu)$ is a direct summand  $\HH_{\Q(v)}(\Sym_m)C'_z$ only if $\mu\unrhd\Shape(\u_z\downarrow_m)=\Shape(P^{[m]}(z))^*$. Similarly, by Corollary \ref{keycor1}, for any $y\in\Sym_m$ with $y\leq_R z$, $S'(\mu)$ is a direct summand  $\HH_{\Q(v)}(\Sym_m)C'_y$ only if $\mu\unrhd\Shape(\u_y\downarrow_m)\unrhd\Shape(\u_z\downarrow_m)$. However, applying Proposition \ref{prop:wTow1-n-m}, we can deduce that $$\begin{aligned}
\Shape(\u_z\downarrow_m)&=\Shape(P^{[m]}(z))^*=\Shape(P^{[n]}(zd)\downarrow_{m})^*=\Shape(P^{[n]}(w)\downarrow_{m})^*\\
&=\Shape(P^{[n]}(w)^*\downarrow_{m})=\Shape(\u_w\downarrow_{m}) .
\end{aligned}
$$

It follows that there exists an integer $r>0$ such that as a left $\HH_{\Q(v)}(\Sym_m)$-module, $\HH_{\Q(v)}(\Sym_m)C'_{w}$ is a direct summand of $$
\Bigl(\bigoplus_{\substack{\mu\vdash m\\ \mu\unrhd\Shape(\u_{w}\downarrow_{m})}} S'(\mu)\Bigr)^{\oplus r},
$$
which is a contradiction to (\ref{sta2}). This completes the proof of Theorem \ref{mainthm3}.\qed

\medskip
\noindent
{\bf Proof of Theorem \ref{mainthm2}}: This follows from Theorem \ref{mainthm3} and Lemma \ref{hm2}.\qed

%
%
%
%
%
%
%

\medskip
\noindent
{\textbf{Proof of Theorem \ref{mainthm1}}}: Let $x,y\in\Sym_n$ with $x\leq_L y$. Then there exists some $h\in\HH_v(\Sym_n)$ such that $C'_x$ appears with a nonzero coefficient in the expansion of $hC'_y$ as a linear combination of the Kazhdan-Lusztig $C'$-bases. Let $\lam\in\P_n$ such that $\RS(y)\in\Std^2(\lam)$. Applying Theorem \ref{mainthm3}, we can get that \begin{equation}\label{Cyfst1}
C'_y= v^{-\ell(w_{\lam^*,0})}\mff_{\u_y,\v_y}+\sum_{\substack{(\s,\t)\in\Std^2(n)\\ (\s,\t)\Gdom (\u_y,\v_y)}}r'_{y,\s,\t}\mff_{\s\t},
\end{equation}
where $r'_{y,\s,\t}\in\Q(v)$ for each triple $(y,\s,\t)$. It follows that \begin{equation}\label{Cyfst2}
hC'_y=\sum_{\u\in\Std(\lam^*)}r_{y,\u}''\mff_{\u,\v_y}+\sum_{\substack{(\s,\t)\in\Std^2(n)\\ \t \unrhd \v_y}}r''_{y,\s,\t}\mff_{\s\t},
\end{equation}
where $r_{y,\u}'', r''_{y,\s,\t}\in\Q(v)$ for each $(y,\u)$ and $(y,\s,\t)$.

By the second equality in Theorem \ref{mainthm3}, for any $(\s,\t)\in\Std^2(\mu)$, we have that \begin{equation}\label{Cyfst3}
\mff_{\s\t}= v^{\ell(w_{\mu^*,0})}C'_{\pi_{\mu^*}(\s^*,\t^*)}+\sum_{\substack{z\in\Sym_n\\ (\u_z,\v_z)\Gdom (\s,\t)}}r'''_{\s,\t,z}C'_z,
\end{equation}
where $r'''_{\s,\t,z}\in\Q(v)$ for each triple $(\s,\t,z)$.

Now combining (\ref{Cyfst2}), (\ref{Cyfst3}) and our assumption that $C'_x$ appears with a nonzero coefficient in the expansion of $hC'_y$ as a linear combination of the Kazhdan-Lusztig $C'$ bases, we can deduce that $\v_x\unrhd\v_y$ and hence $Q(x)\unlhd Q(y)$. This completes the proof of Theorem \ref{mainthm1}.\qed

\bigskip
\section{Twisted Kazhdan-Lusztig bases and dual seminormal bases}

In this section, we shall prove the results on the relation between twisted Kazhdan-Lusztig bases, dual seminormal bases and dual Murphy bases.


\begin{dfn}[{\cite{Mur}}] Let $\lam\in\P_n$. We define $$
\fn_{\lam}:=\sum_{w\in\Sym_{\lam}}(-v)^{-\ell(w)}T_w.
$$
For any $\s,\t\in\Std(\lam)$, we set \begin{equation}\label{nst}
\fn_{\s\t}:=T_{d(\s)}^*\fn_{\lam}T_{d(\t)}. \end{equation}
\end{dfn}

\begin{thm}[{\cite{Mur}}] The set $\{\fn_{\s\t}\mid \s,\t\in\Std(\lam),\lam\in\P_n\}$ forms an $\A$-basis of $\HH_{v}(\Sym_n)$.
\end{thm}

We call $\{\fn_{\s\t}\mid \s,\t\in\Std(\lam),\lam\in\P_n\}$ the {\bf dual Murphy basis} of $\HH_{v}(\Sym_n)$. It is cellular in the sense of Graham and Lehrer \cite{GL}.

\begin{dfn}\label{gdfn2} Let $\lam\in\P_n$. For any $\s,\t\in\Std(\lam)$, we define $$\mfg_{\s\t}:=F_{\s^*}\fn_{\s\t}F_{\t^*}. $$
\end{dfn}

\begin{cor}\label{obobn2} The set \begin{equation}\label{dualsemi}\{\mfg_{\s\t}\mid \s,\t \in \Std(\lam), \lam\in\P_n\}\end{equation} is a basis of $\HH_{\Q(v)}(\Sym_n)$. Moreover,
\begin{enumerate}
\item[1)] If $\s, \t, \u$ and $\v$ are standard tableaux, then $\mfg_{\s\t}\mfg_{\u\v}=\delta_{\t\u}\gamma'_{\t}\mfg_{\s\v}$, where $\gamma'_{\t}\in\Q(v)^{\times}$;
\item[2)] For any $(\s,\t)\in\Std^2(n)$, we have $\mfg_{\s\t}\in\Q(v)^{\times}\mff_{\s^*\t^*}$. In particular, $\mfg_{\s\t}\HH_{\Q(v)}(\Sym_n)\cong S(\lam^*)$,
$\HH_{\Q(v)}(\Sym_n)\mfg_{\s\t}\cong S'(\lam^*)$.
\item[3)] For any $1\leq m\leq n$,  $\mfg_{\s\t}\HH_{\Q(v)}(\Sym_m)\cong S(\mu^*)$,
$\HH_{\Q(v)}(\Sym_m)\mfg_{\t\s}\cong S'(\mu^*)$, where $\mu:=\Shape(\t\downarrow_m)$.
\end{enumerate}
\end{cor}

\begin{proof} Parts 1) and 2) follow from \cite{Ma}. Part 3) follows from Part 2) and Lemma \ref{resm}.
\end{proof}
We call $\{\mfg_{\s\t}\mid \s,\t\in\Std(\lam),\lam\in\P_n\}$ the {\bf dual seminormal basis} of $\HH_{\Q(v)}(\Sym_n)$ corresponding to the dual Murphy basis $\{\fn_{\s\t}\mid \s,\t\in\Std(\lam),\lam\in\P_n\}$ of $\HH_{v}(\Sym_n)$.

\begin{lem}[{\cite{Mur}, \cite{HuMathas:Graded Induction}}]
\label{hm22}
For any $\lam\in\P_n$ and $\s,\t\in\Std(\lam)$, we have $$\begin{aligned}
\fn_{\s\t}&=\mfg_{\s\t}+\sum_{\substack{(\u,\v)\in\Std^2(n)\\ (\u,\v)\Gdom(\s,\t)}}p^{\s\t}_{\u\v}\mfg_{\u\v},\\
\mfg_{\s\t}&=\fn_{\s\t}+\sum_{\substack{(\u,\v)\in\Std^2(n)\\ (\u,\v)\Gdom(\s,\t)}}{\hat{p}}^{\s\t}_{\u\v}\fn_{\u\v},
\end{aligned} $$
where $p^{\s\t}_{\u\v}, {\hat{p}}^{\s\t}_{\u\v}\in\Q(v)$ for each pair $(\u,\v)$.
\end{lem}

\begin{lem}\label{Cseminorma2} Let $\lam\vdash n$. For each $w\in\Sym_n$ with $\RS(w)\in\Std(\lam)\times\Std(\lam)$, we have that $$\begin{aligned}
C_w&= \eps_{w_{\lam^*,0}} v^{\ell(w_{\lam^*,0})}\mfg_{\u_w,\v_w}+\sum_{\substack{(\s,\t)\in\Std^2(n)\\ (\s,\t)\rhd(\u_w,\v_w)}}p'_{w,\s,\t}\mfg_{\s\t},\\
\mfg_{\u_w,\v_w}&= \eps_{w_{\lam^*,0}} v^{-\ell(w_{\lam^*,0})}C_w+\sum_{\substack{y\in\Sym_n\\ (\u_y,\v_y)\rhd(\u_w,\v_w)}}p'_{w,y}C_y,
\end{aligned}$$
where $p'_{w,\s,\t},p'_{w,y}\in\Q(v)$ for each triple $(w,\s,\t)$ and each pair $(w,y)$.
\end{lem}

\begin{proof} The second equality follows from Lemma \ref{hm22}, the first equality in \cite[Corollary 4.11]{Gec} and \cite[Corollary 5.11]{Gec}. The first equality follows from the second equality.
\end{proof}

The following two theorems are the third main result of this paper.

\begin{thm}\label{mainthm4} Let $\lam\vdash n$. For each $w\in\Sym_n$ with $\RS(w)\in\Std(\lam)\times\Std(\lam)$, we have that $$\begin{aligned}
C_w&= \eps_{w_{\lam^*,0}} v^{\ell(w_{\lam^*,0})}\mfg_{\u_w,\v_w}+\sum_{\substack{(\s,\t)\in\Std^2(n)\\ (\s,\t)\Gdom(\u_w,\v_w)}}p'_{w,\s,\t}\mfg_{\s\t},\\
\mfg_{\u_w,\v_w}&= \eps_{w_{\lam^*,0}} v^{-\ell(w_{\lam^*,0})}C_w+\sum_{\substack{y\in\Sym_n\\ (\u_y,\v_y)\Gdom(\u_w,\v_w)}}p'_{w,y}C_y,
\end{aligned}$$
where $p'_{w,\s,\t},p'_{w,y}\in\Q(v)$ for each triple $(w,\s,\t)$ and each pair $(w,y)$.
\end{thm}

\begin{proof} The second equality follows from the first equality. We only prove the first equality. We use induction on $\ell(w)$. If $\ell(w)=0$, then $w=1$.
In this case, $\u_w=\v_w=\t^{(1^n)}$ which is the unique minimal element (under ``$\rhd$'') in $\Std(n)$. So there is nothing to prove.

Now let $w\in\Sym_n$ with $\ell(w)=k>0$. Assume that the  first equality of this theorem holds for any $C_u$ with $\ell(u)<k$. We now consider $C_w$. Suppose that there exists some $(\s,\t)\in\Std^2(n)$ such that  $(\s,\t)\rhd (\u_w,\v_w)$, $p'_{w,\s,\t}\neq 0$, but $(\s,\t)\not\Gdom (\u_w,\v_w)$. We fix such a pair $(\s,\t)$. Then by definition of ``$\Gdom$'', we can find an integer $1\leq m<n$ such that
either $\Shape(\s\downarrow_m)\ntrianglerighteq\Shape(\u_w\downarrow_m)$ or $\Shape(\t\downarrow_m)\ntrianglerighteq\Shape(\v_w\downarrow_m)$.

Suppose that $\Shape(\t\downarrow_m)\ntrianglerighteq\Shape(\v_w\downarrow_m)$. Then applying Corollary \ref{obobn2} and Lemma
\ref{Cseminorma2}, we can deduce that \begin{equation}\label{sta1a}
\begin{matrix}\text{there exists some $\mu\in\P_m$ with $\mu\ntrianglerighteq\Shape(\v_w\downarrow_m)$, such that $S(\mu^*)$ occurs}\\
\text{as a direct summand of $C_w\HH_{\Q(v)}(\Sym_m)$.}\end{matrix}
\end{equation}

On the other hand, we can decompose $w=dz$, where $d\in\dmn^{-1}$, $z\in\Sym_m$, where $1\leq m< n$. Applying the involution $j$ to the equality in Lemma \ref{mucoeff} and an induction on $\ell(d)$, we can deduce that $$
C_{w}=C_{dz}=C_{d}C_z+\sum_{\substack{w>u\in\Sym_n\\ u\leq _L z}}\check{r}_{w,u}C_u,
$$
where $\check{r}_{w,u}\in\A$ for each pair $(w,u)$.

It follows that
$$
C_{w}\HH_{\Q(v)}(\Sym_m)=\Bigl(C_{d}C_z+\sum_{\substack{w>u\in\Sym_n\\ u\leq _L z}}\check{r}_{w,u}C_u\Bigr)\HH_{\Q(v)}(\Sym_m)
$$
is contained in $$
N_1(d,z):=C_dC_z\HH_{\Q(v)}(\Sym_m)+\sum_{\substack{w>u\in\Sym_n\\ u\leq _L z}}\check{r}_{w,u}C_u\HH_{\Q(v)}(\Sym_m) .
$$
There is a natural surjection from $$
N_2(d,z):=C_z\HH_{\Q(v)}(\Sym_m)\bigoplus\Bigl(\bigoplus_{\substack{w>u\in\Sym_n\\ u\leq _L z}}\check{r}_{w,u}C_u\HH_{\Q(v)}(\Sym_m)\Bigr).
$$
onto $N_1(d,z)$. Since $\HH_{\Q(v)}(\Sym_m)$ is semisimple, it follows that $N_1(d,z)$ is a direct summand of $N_2(d,z)$. Hence $C_{dz}\HH_{\Q(v)}(\Sym_m)$ is a direct summand of $N_2(d,z)$.

Applying the first equality in Lemma \ref{Cseminorma2} and Lemma \ref{gdfn2}, we can deduce that for any $\mu\in\P_m$, $S(\mu^*)$ is a direct summand  $C_z\HH_{\Q(v)}(\Sym_m)$ only if $\mu\unrhd\Shape(\v_z)=\Shape(Q(z))^*$. Now $u\leq_L z$ implies that $\v_u\unrhd\v_z$ by Theorem \ref{mainthm1}. Since $u<w$, we can apply the induction hypothesis to $C_u$ to conclude that $S(\mu^*)$ is a direct summand  $C_u\HH_{\Q(v)}(\Sym_m)$ only if $\mu\unrhd\Shape(\v_u\downarrow_m)\unrhd\Shape(\v_z\downarrow_m)$. However, as noted in the proof of Theorem \ref{mainthm3}, we have $\Shape(\v_z\downarrow_m)=\Shape(\v_w\downarrow_{m})$. It follows that there exists an integer $r>0$ such that as a right $\HH_{\Q(v)}(\Sym_m)$-module, $C_{w}\HH_{\Q(v)}(\Sym_m)$ is a direct summand of $$
\Bigl(\bigoplus_{\substack{\mu\vdash m\\ \mu\unrhd\Shape(\v_{w}\downarrow_{m})}} S(\mu^*)\Bigr)^{\oplus r},
$$
which is a contradiction to (\ref{sta1a}).

Suppose that $\Shape(\s\downarrow_m)\ntrianglerighteq\Shape(\u_w\downarrow_m)$. The proof is completely similar by considering the left $\HH_{\Q(v)}(\Sym_m)$-module generated by $C_w$ instead of the right $\HH_{\Q(v)}(\Sym_m)$-module. This completes the proof of the theorem.
\end{proof}

\begin{thm}\label{mainthm5} Let $\lam\vdash n$. For each $w\in\Sym_n$ with $\RS(w)\in\Std(\lam)\times\Std(\lam)$, we have that $$\begin{aligned}
C_w&= \eps_{w_{\lam^*,0}} v^{\ell(w_{\lam^*,0})}\fn_{\u_w,\v_w}+\sum_{\substack{(\s,\t)\in\Std^2(n)\\ (\s,\t)\Gdom(\u_w,\v_w)}}p''_{w,\s,\t}\fn_{\s\t},\\
\fn_{\u_w,\v_w}&= \eps_{w_{\lam^*,0}} v^{-\ell(w_{\lam^*,0})}C_w+\sum_{\substack{y\in\Sym_n\\ (\u_y,\v_y)\Gdom(\u_w,\v_w)}}p''_{w,y}C_y,
\end{aligned}$$
where $p''_{w,\s,\t},p''_{w,y}\in\Z[v,v^{-1}]$ for each triple $(w,\s,\t)$ and each pair $(w,y)$.
\end{thm}

\begin{proof} 
The first equality follows from 
Lemma \ref{Cseminorma2}
and Theorem \ref{mainthm4}.
The second equality follows from the first equality.
\end{proof}

\end{document}